\documentclass[12pt,leqno]{amsart}
\usepackage{amsmath,amsfonts,amssymb,amsthm}
\usepackage{graphicx}
\graphicspath{ }
\usepackage{wrapfig}
\usepackage{accents}
\usepackage{caption}
\usepackage{subcaption}
\usepackage{calligra}
\usepackage[colorlinks]{hyperref}
\hypersetup{citecolor=black}

\numberwithin{figure}{section}

\theoremstyle{plain}
\newtheorem{thm}{Theorem}[section]
\newtheorem{lem}[thm]{Lemma}

\newtheorem{cor}{Corollary}[thm]
\theoremstyle{definition}
\newtheorem{defn}{Definition}[section]

\newtheorem{exmp}{Example}[section]

\theoremstyle{remark}

\usepackage{mathtools}
\title{\textbf{On convergence and rational summation of power series in p-adic field }}
\author[A. A. Shaikh]{Absos Ali Shaikh$^{1}$}
\address{Department of Mathematics,\\ The University of Burdwan,\\ \newline Burdwan-713101, West Bengal, India.}
\email{$^1$aask2003@yahoo.co.in, aashaikh@math.buruniv.ac.in}
\author[M. A. Sarkar]{MABUD ALI SARKAR$^2$}
\address{Department of Mathematics\\ The University of Burdwan \\ Burdwan-713101, India.}
\email{$^2$mabudji@gmail.com}

\begin{document}

\begin{abstract}
	In this paper we have discussed convergence of power series both in p-adic norm as well as real norm. We have investigated rational summability of power series with respect to both p-adic norm and real norm under certain conditions. Then we have studied convergence of specially constructed power series and derived summation formula. Finally, we have studied the adele, idele and some results regarding it with the help of convergent power series.
\end{abstract} 
\footnotetext{
	$\mathbf{2010}$\hspace{10pt}Mathematics\; Subject\; Classification: 12J25, 32P05, 26E30, 40A05, 40D99, 11S99.}
\maketitle
\section{Introduction} 
For an ordered field $ F $, the  Archimedian property says that for any positive $a $ and $b$ belonging to $F$ there exists a natural number $ n$ such that $ na>b$ i.e., $\underbrace{a+ \cdots+ a}_{n \ \text{times}}>b $. If the Archimedian property does not hold, then the field is called non-Archimedean valued field. In this paper we have considered a non-ordered non-Archimedian valued field, particularly, the p-adic field denoted by $ \mathbb{Q}_p $, which is a non-ordered and non-Archimedian valued field. \\ Our main interests are-
\begin{enumerate}
	\item To investigate power series having rational sum both in p-adic norm as well as real norm, \item What are the rational numbers at which the power series converges with respect to both usual absolute value as well as p-adic absolute value,  \item To study adeles and ideles with the help of rational summable power series in p-adic norm as well as real norm. \end{enumerate} 
To this end we have  included two Lemmas (4.1) and (4.2) to prove the Theorem (4.3) in which we have show equal radius of convergence of a power series with respect to both p-adic absolute value and usual absolute value and then characterised the rationals within the common radius of convergence. Then we have investigated a power series with rational sum both in p-adic norm as well as real norm. Next we have proved two Lemmas (4.5) and (4.6) and using these Lemmas we have proved Theorem (4.7) for a power series having rational sum under certain conditions with respect to both p-adic norm as well as real norm. Next we have discussed convergence of a special power series in Theorem (4.8) and using this power series we have derived a summation formula in Theorem (4.9), giving sum in p-adic norm. Finally we have studied Adelic aspect i.e., Adele and Idele with the help of p-adic convergent power series in the Theorems (4.11), (4.12) and (4.13) using (4.10). \\ \\
The paper has been oriented in the following way: Section $2$ deals with some basic facts of the construction of p-adic field $ \mathbb{Q}_p$. Section $3$ is devoted to the study of p-adic valuation and region of convergence in p-adic fields. Section $4$ is concerned with results and their proofs. Finally conclusion is given.
\section{Some basic facts about p-adic numbers and p-adic fields}
	This section deals with the construction of p-adic field $ \mathbb{Q}_p$ just like obtaining real field $ \mathbb{R} $ as the completion of rational field $ \mathbb{Q}$. For this purpose, the concept of absolute value on $ \mathbb{Q} $ is defined as follows:  
\begin{defn} (\cite{kj})
	The usual absolute value on $ \mathbb{Q} $ is a mapping $ |.| : \mathbb{Q} \to \mathbb{R}$ given by 
	 $$|x|= \left\{
	 \begin{array}{ll}
	 	x  \ \ : x \geq 0, & \hbox{} \\ -x : x <0
	 	 & \hbox{}
	 \end{array}
	 \right. .$$
\end{defn} 
satisfying \\  (i) $ |x|=0 \ \text{if and only if } \  x=0 $ \\ (ii) $ |xy|=|x||y| $ \\ (iii) $ |x+y| \leq |x|+|y| \ \text{(Triangle Inequality)} .$ \\ The usual absolute value $ |.| $ define a metric $ d: \mathbb{Q} \times \mathbb{Q} \to \mathbb{R} $ on $\mathbb{Q}$  by $ d(x,y)=|x-y|$ for all $x,y \in \mathbb{Q}$ and with respect to this metric $ \mathbb{Q}$ is a metric space. A Cauchy sequence in a metric space is a sequence where distances between two consecutive terms decreases. A metric space is complete if every Cauchy sequence in this space converges in it. The field $ \mathbb{Q}$ is not complete with respect to the usual absolute value i.e., $  (\mathbb{Q},|.|)  $ is not a complete metric space. For, the well-known sequence $ \{1,1.4,1.41,1.414, \cdots \} $ which is a Cauchy sequence in $ \mathbb{Q}  $ converging to $  \sqrt{2}  $  but $  \sqrt{2} \notin \mathbb{Q}$. Now we complete $ \mathbb{Q}  $ with respect to the usual absolute value $  |.|  $ which yields $ \mathbb{R} $. This is obtained by including all possible limits of every Cauchy sequences in $ \mathbb{Q}$  with the space $(\mathbb{Q},|.|)$. This is the way we get $ \mathbb{R}  $ from $ \mathbb{Q}  $. Since completion of any field is again a field, the completion $ \mathbb{R}$ is also a field. \\ Now let us consider a different absolute value on $\mathbb{Q}$ other than the usual absolute value. Then what do we get? The answer is given by the following:
\begin{defn}\cite{kj}
	For every prime $ p $ we get an p-adic absolute value on $ \mathbb{Q}$ associated with the corresponding $p$. Fix a prime $p$ and choose a non-zero  number $x \in \mathbb{Q} $. Then $ x$ can be represented as $$ x=p^n \frac{a}{b},$$  for $b \neq 0$ and $ p$ does not divide $ a,b$. Then the p-adic absolute value on $\mathbb{Q}$ is defined by $$ |x|_p=|p^n \frac{a}{b}|_p=p^{-n}. $$ For $0 \in \mathbb{Q}$, we define $ |0|_p =0$. We can easily notice that $ |.|_p $ maps into the discrete set $ \ \{p^n: n \ \text{is integer} \} \cup \{0\}. $
\end{defn} 
\begin{exmp}
	Some examples of p-adic absolute values on $ \mathbb{Q} $ are given below: 
	$$ |25|_5=|5^2|_5=5^{-2}=\frac{1}{25}, $$
	$$ |100|_2=|2^2 \times 5^2|_2=2^{-2}=\frac{1}{4}, $$
	$$ |\frac{196}{5}|_7=|7^2 \cdot \frac{4}{5}|_7=7^{-2}=\frac{1}{49}, \ etc. $$ 

\end{exmp}
From the p-adic absolute value we can define(\cite{ii}) a metric $d: \mathbb{Q} \times \mathbb{Q} \to \mathbb{Q}_{+}$, where $ \mathbb{Q}_+ $ denotes the positive rationals, on $ \mathbb{Q} $ by $$ d(x,y)=|x-y|_p.$$ 
The p-adic metric measures distance in an unexpected way. For example, 
$$ |6879-4|_5=|6875|_5=|5^4 \times 11|_5=5^{-4}=\frac{1}{625},$$
$$ |5-4|_5=|1|_5=|7^{0} \times 1|_5=5^0=1.$$
Thus, we have 
$ |6879-4|_5<|5-4|_5$, which implies that p-adic distance between two distant points becomes less than that between two nearer points. The important fact is that $ (\mathbb{Q}, |.|_p) $ is not complete. But just like the process of completion of $ \mathbb{Q} $ with respect to the usual absolute value $  |.| $, we make completion of $ \mathbb{Q}$ with respect to $|.|_p$. 
For every prime number $p$, we get a completion of $ \mathbb{Q}$ with respect to the associated p-adic metric. This completion is named as $  \mathbb{Q}_p $. Since $ \mathbb{Q} $ is a field, its completion $ \mathbb{Q}_p $ is also a field, called the field of p-adic numbers or p-adic field. The p-adic absolute value $  |.|_p $  satisfies the following: \\
(i) $ \ |x|_p=0 \ \text{if and only if} \ x=0 , $ \\ 
(ii) $ \ |xy|_p=|x|_p |y|_p , $ \\
(iii) $ \ |x+y|_p \leq \max \{|x|_p , |y|_p \} . $ \\ 
We note that the condition $ (iii) $ also satisfies the triangle inequality as 
$$ |x+y|_p \leq \max \{|x|_p, |y|_p\} \leq |x|_p+|y|_p . $$
Actually the condition $(iii)$ is much stronger than that of triangle inequality and this is called strong triangle inequality or ultrametric triangle inequality.  
\\ 
The p-adic absolute value $  |.|_p $ is also called non-Archimedian absolute value because for any any integer $ n $, 
\begin{align*} 
	|n|_p=\underbrace{|1+1+\cdots +1|_p}_{n \ times} & \leq  \max \{|1|_p,|1|_p, \cdots |1|_p \} =|1|_p=1, \\ 
	i.e., \ \  & |n|_p \leq 1.
\end{align*} 
Hence the field $ \mathbb{Q}_p $ is a non-Archimedian valued field. We know that all non-zero squares of an ordered field is positive. But the number $\sqrt{-7}$ exist in $ \mathbb{Q}_2$ and its square is $-7$, which is not positive. Hence $ \mathbb{Q}_2$ is non-ordered field and in fact $\mathbb{Q}_p$ is non-ordered field for every prime $ p$. Thus $\mathbb{Q}_p$ is a non-ordered and non-Archimedean field.
\section{p-adic valuation and region of convergence in p-adic fields}
First we define p-adic valuation $ \text{ord}_p $ by $$ \text{ord}_p(x)=\max \{r : p^r |x \} , \ x \in \mathbb{Q}. $$ 
For non-zero $ x \in \mathbb{Q} $, we have $$ x=p^k \frac{a}{b} , \ b \neq 0 , p \ \text{does not divide  } a,b.$$
In view of the above definition of p-adic valuation of $ x $, we have 
$$ \text{ord}_p(x)=k. $$
We can simply relate this fact with p-adic absolute value  by 
$$ |x|_p=p^{- \text{ord}_p(x)}.$$
\begin{exmp} 
Some examples of p-adic valuation are $$ \text{ord}_5(25)=2, \ \ \text{ord}_3(2)=0 , \ \ \text{ord}_7 \left(\frac{50}{49} \right)=-2. $$
The p-adic valuation satisfies (\cite{ii} \cite{jgh}) 
\\ 
(i) $ \text{ord}_p(x) =\infty \Leftrightarrow x=0 $, \\ 
(ii) $ \text{ord}_p(xy)=\text{ord}_p(x)+\text{ord}_p(y) \  $ , \\
(iii) $ \text{ord}_p \left(\frac{x}{y}\right)=\text{ord}_p(x)-\text{ord}_p(y) \ $, \\
(iv) $ \text{ord}_p(x+y) \geq \min \{ \text{ord}_p(x), \text{ord}_p(y)\} $.
\end{exmp}
Note that convergence of power series is a topological issue. Just like real or complex analysis, the Hadamard's formula ( \cite{hj} \cite{j} \cite{jhg}  ) for radius of convergence of a power series holds good in non-Archimedian geometry and is stated as follows: \\ Let $ f(x)=\sum_{n=1}^{\infty} a_nx^n  \ \text{where} \ a_n \in \mathbb{Q}_p $. Define $ R \in [0, \infty] $ by the formula $$ \frac{1}{R}=\varlimsup_{n \to \infty} \sqrt[n]{|a_n|_p}, $$
where $ \frac{1}{0}=\infty \  \text{and} \ \frac{1}{\infty}=0 $. For $ x \in \mathbb{Q}_p $, $  f(x) $ converges if $  |x|_p <R  $ and $  f(x) $ diverges if $ \ |x|_p>R \ $. 
\\
Now the natural question arises what can we say about the convergence of $ f(x) $ at $ x \in \mathbb{Q}_p $ when $  |x|_p=R \ , R>0$ ? The answer is given below: \\
Since $ \ |.|_p $ is non-Archimedian, the series $ f(x)=  \sum_{n=1}^{\infty} a_n x^n $ in $ \ \mathbb{Q}_p $ converges if and only if $ \ |a_nx^n|_p=|a_n|_p R^n \to 0 \ \text{as} \ n \to \infty $ and this depends on $  x  $. Thus $ f(x) $ may converge at $  x \in \mathbb{Q}_p $ with $ |x|_p=R $ or $ f(x) $ may not converge at $  x \in \mathbb{Q}_p $ with $  |x|_p=R  $. Hence the region of convergence of $ f(x)=\sum_{n=1}^{\infty} a_n x^n \ $ is  given by $ \{x \in \mathbb{Q}_p : |x|_p<R \}  \ \text{or} \ \{x \in \mathbb{Q}_p: |x|_p \leq R \}. $ \\ Now the main questions are as follows: \begin{enumerate}
	\item  Does there exist a power series having rational sum in p-adic norm as well real norm ? \item Is it possible to derive a summation formula resulting rational sum of a power series in p-adic norm as well as real norm? \item What are the rationals at which the power series converges both in p-adic norm as well real norm within the equal radius of convergence ? \item Does there exist any relation among p-adic convergent power series, Adle and Idele?
\end{enumerate}
Well, we have attacked and answered each of these questions in the following section of $\text{Main Results}$.

\section{Main Results}

We consider the generalised binomial expansion $$B(b,x)=(1+x)^b=\sum_{n=0}^{\infty} \binom{b}{n}x^n=\sum_{n=0}^{\infty} \frac{b(b-1) \cdots (b-n+1)}{n!}x^n,$$ where $b$  is either real or complex number. The series converges or diverges according as $|x|<1$ or $|x|>1$ respectively, with respect to the usual absolute value. \\ Next, we note that $\mathbb{Z}_p$ is the ring of p-adic integers, which is a completion of $\mathbb{Z}$ with respect to p-adic absolute value. The algebraic closure of $ \mathbb{Q}_p$ is denoted by $ \bar{\mathbb{Q}}_p$, the completion of $ \bar{\mathbb{Q}}_p$ is denoted by $\Omega_p$, which is a nice field in p-adic analysis and $\Omega_p[[X]]$ is denoted as the ring of formal power series over $\Omega_p$.  \\ Let us define the open disc with radius $r$ centered at $0$ by  $ D(r^{-})=\{x \in \Omega_p: |x|<r \}$. 
For $b \in \Omega_p$, define the follwing function: \\
$$ B(b,p,X)=\sum_{n=0}^{\infty} \frac{b(b-1) \cdots (b-n+1)}{n!}X^n \in \Omega_p[[X]],$$
which is similar of generalised binomial function. This may be termed as formal power series because we still do not know about its convergence and only know about its coefficients. Now we check the convergence of the power series above. We will consider two cases $|b|_p >1$ and $|b| \leq 1$ differently. If $|b|_p>1$, then $|b-j|_p=|b|_p, \ \forall j=0,1,2,3, \cdots$ and so the $|.|_p$ value of the $n^{th}$ term of the power series will be 
\begin{align*}
\left| \frac{b(b-1) \cdots (b-n+1)}{n!}X^n \right|_p =\frac{|b|_p |b-1|_p \cdots |b-n+1|_p}{|n!|_p}|X^n|_p &=\frac{|b|_p|b|_p \cdots |b|_p \ (n \ times)}{|n!|_p}|x^n|_p \\ &=\frac{|bx|_p^n}{|n!|_p}=c_n, \ say. \end{align*} If $r$ is the radius of convergence, then by Cauchy-Hadamard's formula, we have \\ \begin{align*}
 r= \frac{1}{\lim_{n \to \infty} \sqrt[n]{|c_n|_p}}=\frac{1}{\lim_{n \to \infty} \sqrt[n]{\left|\frac{b^n}{n!}\right|_p}} &=\frac{1}{\lim_{n \to \infty} |b|_p \sqrt[n]{|n!|_p}} =\frac{p^{-\frac{1}{p-1}}}{|b|_p}.  \end{align*} 
Thus the region of convergence of the power series $B(b,p,X)$ is $D(r^{-})=D \left( (\frac{p^{-\frac{1}{p-1}}}{|b|_p})^{-} \right)$, which depends on the number $b$. \\ 
If $ |b|_p \leq 1$ i.e., $ b \in \mathbb{Z}_p$, then we have $ |b-j|_p \leq 1, \ \forall j=0,1,2,3, \cdots$, and therefore $$ \left|\frac{b(b-1) \cdots (b-n+1)X^n}{n!} \right|_p \leq \left|\frac{X^n}{n!} \right|_p. $$
Then proceeding as above the radius of convergence of the right-hand series $ \sum_{n=0}^{\infty} \frac{X^n}{n!}$ is equal to $p^{-\frac{1}{p-1}}$ with respect to the p-adic absolute value. Hence the power series $B(b,p,X)$ converges at least on $D \left((p^{-\frac{1}{p-1}})^{-} \right)$ but still we have to investigate more exact result on the region of convergence of $B(b,p,X)$ where $b \in \mathbb{Z}_p$. \\ At this end, we considered the following Lemma which will be used in Theorem (4.3).
\begin{lem}
	\cite{abc} Each $f(X) \in \mathbb{Z}_p[[X]]$ converges in $D(1^{-})$, where $ .\mathbb{Z}_p[[X]]$ is a ring of formal power series.
\end{lem} 
We claim that $B(b,p,X) \in \mathbb{Z}_p[[X]].$ For this we need to show that the coefficient $ \frac{b(b-1) \cdots (b-n+1)}{n!} \in \mathbb{Z}_p$. This is achieved from the following Lemma which will be used in Theorem (4.3) and (4.14). 
\begin{lem}
If $b \in \mathbb{Z}_p$, then $ \binom{b}{n}=\frac{b(b-1) \cdots (b-n+1)}{n!} \in \mathbb{Z}_p.$	
\end{lem}
  \begin{proof} We give this proof in our way as follows.\\ Since $ \mathbb{Z} \subset \mathbb{Z}_p$, then either $b \in \mathbb{Z}$ or $ b \in \mathbb{Z}_p \setminus \mathbb{Z}$. \\ If \ $b \in \mathbb{Z}$, then construct the function $f(b)=\frac{b(b-1) \cdots (b-n+1)}{n!}=\binom{b}{n}$, which represents the number of distinct combinations of $n$ things chosen from a set of $b$ objects and indeed this is  a positive integer and hence a p-adic integer i.e., $ \frac{b(b-1) \cdots (b-n+1)}{n!} \in \mathbb{Z}_p.$
 \\ Next look at the fact that the function $f(b)=\frac{b(b-1) \cdots (b-n+1)}{n!}$ defined above is a continuous function from $\mathbb{Z}_p$ to $\mathbb{Z}_p$. Also the ring of p-adic integers $ \mathbb{Z}_p$ is closed with respect to the p-adic topology obtained from p-adic absolute value. Then $f(b_i)=\frac{b_i(b_i-1) \cdots (b_i-n+1)}{n!} \in \mathbb{Z}.$ Now for $ b \in \mathbb{Z}_p \setminus \mathbb{Z}$, construct a sequence  of integers $ \{b_i\}_{i \in \mathbb{Z}^{+}} $ in $ \mathbb{Z}$ with limit $ b \in \mathbb{Z}_p$. Since $ f$ is continous, we have $ f(b)=\lim_{i \to \infty} (f(b_i) )$. Since $ \mathbb{Z}_p$ is closed with respect to the p-adic topology, we conclude that the limit function $f(b)$ belongs to $ \mathbb{Z}_p$. So for all $b \in \mathbb{Z}_p$, the coefficient $ \frac{b(b-1) \cdots (b-n+1)}{n!} \in \mathbb{Z}_p$. \end{proof}
 To this end we have the following theorem which will be used in Theorem (4.7): 
\begin{thm} (i)
	Let $|.|$ $( \text{or} \ |.|_{\infty})$ and $|.|_p$ be respectively the usual absolute value and p-adic value on $ \mathbb{Q}$ and consider the power series $ \sum_{n=0}^{\infty} \frac{b(b-1) \cdots (b-n+1)}{n!}x^n$. Then the radius of convergence is equal to $1$ with respect to the usual (resp. p-adic) absolute value on $ \mathbb{Q}$ if $ b \in \mathbb{Q} \cap (\mathbb{Z}_p \setminus Z)$. \\ $(ii)$ The set of rationals $x \in \mathbb{Q}$ such that the above power series converges both in p-adic absolute value and usual absolute value within the same radius of convergence $ 1$ i.e., belonging to both the sets $ \{x \in \mathbb{Q}: |x|_{\infty}< 1 \}$ and $ \{x \in \mathbb{Q}: |x|_p<1 \}$ in which the above power serise converges is given by the set $$\{x \in \mathbb{Q}: x=\frac{u}{v}, \ u,v \in \mathbb{Z}, \ (u,v)=1, \ |u|_{\infty} \leq |v|_{\infty}, \ p|u, \ p \nmid v \}.$$
\end{thm} 
\begin{proof} (i)
 Let us write the coefficient $a_n=\frac{b(b-1) \cdots (b-n+1)}{n!}$.\\
If $R$ be the radius of convergence with respect to the usual value on $ \mathbb{Q}$, then by D'Alembert Ratio test 
\begin{align*}
\frac{1}{R}=\lim_{n \to \infty} \left| \frac{a_{n+1}}{a_n}\right| &=\lim_{n \to \infty} \left| \frac{b(b-1) \cdots (b-n+1)(b-n)}{(n+1)!} \times \frac{n!}{b(b-1) \cdots (b-n+1)} \right| \\ &=\lim_{n \to \infty} \left|\frac{b-n}{n+1} \right|=1. 
\end{align*}
Note that the condition $ b \notin \mathbb{Z}$ is assumed because if $b \in \mathbb{Z}$ then the coefficients $a_n \to 0$ as $ n \to \infty$ and hence in that case the radius of convergence is $ \infty$. \\ 
Again if $b \in \mathbb{Z}_p$, then by Lemma (4.2), $B(b,p,x) =\sum_{n=0}^{\infty} \frac{b(b-1) \cdots (b-n+1)}{n!}x^n \in \mathbb{Z}_p[[x]]$. Hence by the $ \text{Lemma 4.1}$, we conclude that the power series converges in $D(1^{-})$. Thus the radius of convergence with respect to the p-adic absolute $|.|_p$  i.e., on the p-adic field is equal to $1$. Therefore the radius of convergence of the power series is same with respect to both usual absolute value and p-adic absolute value. \\ (ii) \ In this part we have to characterise the rationals $x=\frac{u}{v}$, $ u,v \in \mathbb{Z}$ with $(u,v)=1$ such that the series converges as a series in $ \mathbb{R}$ as well as a series in $ \mathbb{Q}_p$. If the series convrges, it is necessary and sufficient by the very definition of the radius of convergence that \begin{eqnarray}
|x|_{\infty} <1 \\ |x|_p<1
\end{eqnarray}	
The condition $(1)$ implies $ \left|\frac{u}{v}\right|_{\infty}<1 $ i.e., $ |u|_{\infty} \leq |v|_{\infty}$ while the condition $(2)$ implies $ p|u$ but $p \nmid v$.	 Therefore, if $\sum_{n \geq 0} a_n x^n \in \mathbb{Q}[[x]]$ converging for $|x|_{\infty}<1$ and $|x|_p<1$, then the $ x \in \mathbb{Q}$ such that the series converges both in $\mathbb{R}$ and $ \mathbb{Q}_p$ within radius of convergence $1$ are characterised by $ x =\frac{u}{v}$ with $u,v \in \mathbb{Z}$ and $u,v$ are co-prime numbers with $p|u$ and $p \nmid v$. Therefore the rationals are characterised by the set $ \{x\in \mathbb{Q}: x=\frac{u}{v}, u,v \in \mathbb{Z}, (u,v)=1, \ |u|_{\infty} \leq |v|_{\infty}, \ p|u, p \nmid v  \}$.
This completes the theorem. 
\end{proof}

Next we proved the following important Lemma which will be used in Theorem (4.7), which are something like that was hinted by Professor Neal Koblitz, University of Washington. 
\begin{lem}
If there is an identity between two power series in $\mathbb{Q}[x]$ relating by addition,  multiplication and substitution, then it is an identity in the ring of formal power series $\mathbb{Q}[[X]]$.
\end{lem}
\begin{proof}
	Let $f(x)=\sum_{n \geq 0}a_nx^n$ and $g(x)=\sum_{n \geq 0} b_n x^n$ be two power series in $\mathbb{Q}$. Consider the identity $f(x)=g(x)$ i.e., $\sum_{n \geq 0} a_nx^n=\sum_{n \geq 0}b_nx^n$. We will show that the same identity holds in the ring of formal power series. That is, we have to prove that $ \sum_{n \geq 0} a_nX^n=\sum_{n \geq 0}b_n X^n$. Let $h(x)=f(x)-g(x)=\sum_{n \geq 0}(a_n-b_n)x^n=0=\sum_{n \geq 0} c_nx^n$, say. Now Taylor series of an analytic function is unique and its coefficients are given by $c_n=\frac{h^{(n)}(0)}{n!}$. As $h(x)=0$, we must have $c_n=0, \forall n$. This implies $a_n-b_n=0, \forall n$ which in turn implies $a_n=b_n , \forall n$. Therefore the identity $ \sum_{n \geq 0}a_n X^n=\sum_{n \geq 0} b_n X^n$ follows immediately in the ring of formal power series.    
\end{proof} 
Now we proved the following Lemma which will be used as tool in Theorem (4.7), (4.11) and (4.12).   
\begin{lem}
	If $u,v$ are coprime integers and if $p$ does not divide $v$, then the condition $u^N \equiv v^N$ (mod $p$) for $N \in \mathbb{N}$ is equivalent to the condition $ \left| \left(\frac{u}{v}\right)^N-1\right|_p<1.$ 
	\begin{proof}
		Let $u,v$ be coprime integers with $p \nmid v$. Then p-adic norm of $v$ is equal to $1$ i.e., $|v|_p=1$. Now, \\
		$ \left| \left(\frac{u}{v}\right)^N-1\right|_p<1 \\ \Leftrightarrow |u^N-v^N|_p<|v|_p^N \\ \Leftrightarrow |u^N-v^N|_p<1, \ (\because |v|_p=1), \\ \Leftrightarrow p | (u^N-v^N) \\ \Leftrightarrow u^N \equiv v^N \ \text{(mod \ p)}. $ \\ This completes the proof.
	\end{proof}
\end{lem}
Now we consider the power series considered in Theorem $(4.3)$ given by $$B(b,p,X)=\sum_{n=0}^{\infty} \frac{b(b-1) \cdots (b-n+1)}{n!}X^n\in \mathbb{Z}_p[[X]], \ b \in \mathbb{Z}_p . $$ Now we prove a theorem as follows:
\begin{thm}
	If $b$ is rational number specially $b=\frac{1}{N}, \ N \in \mathbb{Z}$, $p \nmid N$ and if $X=\left(\frac{u}{v}\right)^N-1, \ u,v, N$ are integers with $ (u,v)=1, \ p \nmid v$ and $ u^N \equiv v^N$ (mod $p$), then the series of rational numbers $B(b,p,X)= \sum_{n=0}^{\infty} \frac{\frac{1}{N}(\frac{1}{N}-1) \cdots (\frac{1}{N}-n+1)}{n!}X^n$ converges to a rational number $\mathbb{Q}_p$ while it converges to a rational number in $\mathbb{R}$ if $ X= \left|\left(\frac{u}{v}\right)^N-1\right|<1, \ v \neq 0$.
\end{thm}
\begin{proof} 
	Let $b=\frac{1}{N}$, $N \in \mathbb{Z}$, $p \nmid N$. Let $x \in D(1^{-})$, then the identity $\left[ B(\frac{1}{N},p,x) \right]=(1+x)^{\frac{1}{N}}$ holds. By Lemma $(4.5)$ the same identity holds in the ring of formal power series as follows: $$ B(\frac{1}{N},p,X)=(1+X)^{\frac{1}{N}}.$$
	So for rationals $b=\frac{1}{N}$, we can write $$ B(b,p,X)=(1+X)^b=\sum_{n=0}^{\infty} \frac{b(b-1) \cdots (b-n+1)}{n!}X^n.$$ 
	Now according to Lemma (4.6), the conditions $(u,v)=1, \ p \nmid v$ and $ u^N \equiv v^N$ (mod $p$) ensures that $ \left| X=\left(\frac{u}{v}\right)^N-1\right|_p<1$ and hence the given power series converges by Theorem (4.3).\\
	Thus to find the sum of the series $ \sum_{n=0}^{\infty} \frac{\frac{1}{N}(\frac{1}{N}-1) \cdots (\frac{1}{N}-n+1)}{n!}X^n$, we just have to find the value of $(1+X)^{\frac{1}{N}}$. Now if $x=\left(\frac{u}{v}\right)^N-1, \ u,v, N$ are integers with $ (u,v)=1$, then $$ \sum_{n=0}^{\infty} \frac{\frac{1}{N}(\frac{1}{N}-1) \cdots (\frac{1}{N}-n+1)}{n!}X^n=(1+X)^{\frac{1}{N}}=\left[ 1+\left(\frac{u}{v} \right)^N -1\right]=\left(\frac{u}{v}\right)^N \in \mathbb{Q},$$
	since $ u,v$ are integers with $(u,v)=1$. This holds in p-adic norm for all primes $p$. The case in real norm follows similarly. 
	\end{proof} 
Now we will prove another theorem involving the binomial series and having rational summation under certain condition.
\begin{thm}
	Let $v \in V=\{2,3, \cdots, p, \cdots \}$ and define $$F_v(x)=\sum_{n=0}^{\infty} \binom{b}{n}x^n=\sum_{n=0}^{\infty} \frac{b(b-1) \cdots (b-n+1)}{n!}x^n, \ x \in \mathbb{Q}_v,$$  with $|x|_v<1$, $b=\frac{r}{s} \in \mathbb{Q}$ and $|b|_v \leq 1$. If for a fixed prime $v'$ and chosen $x \in \mathbb{Q}_{v'}$, $F_{v'}(x)=w \in \mathbb{Q}$, then $F_v(x) \in \mathbb{Q}$ for all prime $v$ if $| \pm c-1|_v<1$ is satisfied.  
\end{thm}	
\begin{proof}
	Given $v \in V=\{2,3, \cdots, p, \cdots \}, \ b=\frac{r}{s}$ with $|b|_v \leq 1$ i.e., $b \in \mathbb{Z}_v$. \\ Since $ b \in \mathbb{Z}_v$, by Lemma (4.2), we conclude $ \binom{b}{n} \in \mathbb{Z}_v$. By Theorem (4.3), $F_v(x)$ converges for $|x|_v<1$. Fix a prime and choose a fixed $x \in \mathbb{Q}_p$, then $y=F_p(x)$ is the unique element $ y \in \mathbb{Q}_p$ such that $y^s=(1+x)^r$ satisfying $|y-1|_v<1$. Here $y$ is unique because it is the unique root of $Y^s=(1+x)^r$ in the disc $\{y: \ |y-1|_v<1 \}$ as in the case of binomial series there exists such an unique root $y$ while the rest roots are of the form $Y=y \zeta_v$, where $\zeta_v$ is some $s^{th}$  root of unity. \\ Now given that for the prime $v' \in V$, we have $F_{v'}(x)=c \in \mathbb{Q}$. Then for every prime $v \in V$ satisfying $|x|_v<1$, then is some $ s^{th}$ root of unity $\zeta_v$ such that $$ F_v(x) \zeta_v=c \in \mathbb{Q} \ \text{and} \ \zeta_v=1 \ \text{if} \ |\pm c-1|_v<1.$$ $$ i.e., \ F_v(x)=c \in \mathbb{Q} \ \text{if} \ |\pm c-1|_v<1.$$ Therefore for all $v \in V$,  $ F_v(x)=\sum_{n=0}^{\infty} \binom{b}{n}x^n=c \in \mathbb{Q} $ if $|\pm c-1|_v<1$. 
	 
\end{proof}

Next we constructed a power series and showed that the power series convreges for all $x \in \mathbb{R}$ and all $x \in \mathbb{Q}_p$ for every  prime $p$. The newly constructed power series is the following \begin{align}
	\phi_{\gamma, \delta}^{\epsilon, q}=\sum_{n=0}^{\infty} \epsilon^n \frac{((\gamma n+\delta)!)^{\gamma n+\delta}}{q+((\gamma n+\delta)!)^{N(\gamma n+\delta})} \frac{x^{\gamma n+\delta}}{(\gamma n+\delta)!}.
\end{align} 
The next theorem highlights the above fact.
\begin{thm}
	The power series $\phi_{\gamma, \delta}^{\epsilon, q}=\sum_{n=0}^{\infty} \epsilon^n \frac{((\gamma n+\delta)!)^{\gamma n+\delta}}{q+((\gamma n+\delta)!)^{N(\gamma n+\delta})} \frac{x^{\gamma n+\delta}}{(\gamma n+\delta)!}$, where $\epsilon=\pm 1$, $0<q \in \mathbb{Q}$, $\gamma \in \mathbb{N}$, $\delta \in \mathbb{N} \cup \{0 \}$ converges for all $x \in \mathbb{R}$ and for all $x \in \mathbb{Q}_p$ for every prime $p$.
\end{thm}
\begin{proof}
	For real case, 
	\begin{align}
		\lim_{n \to \infty} I_{\gamma n +\delta}^{(q)}=\lim_{n \to \infty} \frac{((\gamma n+\delta)!)^{\gamma n+\delta}}{q+((\gamma n+\delta)!)^{N(\gamma n+\delta})}=0.
	\end{align} Thus the power series is convergent at all real $x$. \\ For p-adic case, 
	\begin{align}
		\left|\epsilon^n I_{\gamma n+\delta}^{(q)} \frac{x^{\gamma n+\delta}}{(\gamma n+\delta)!} \right|_p=\frac{\left| (\gamma n+\delta)! \right|_p^{\gamma n+\delta -1}}{\left| q+((\gamma n+\delta)!)^{N(\gamma n+\delta)}\right|_p}|x|_p^{\gamma n+\delta}.
	\end{align}
	By strong triangle inequality of p-adic norm, 
	\begin{align}
		|q+((\gamma n+\delta)!)^{N(\gamma n+\delta)}|_p=|q|_p
	\end{align} for large enough $n$. Now we know that 
	\begin{align}
		|n!|_p=p^{-\frac{n-s_p(n)}{p-1}}, 
	\end{align}
	where $s_p(n)$ is the sum of the base $p$ digits of $n$. According to $(7)$, the numerator of $(5)$ becomes 
	\begin{align}
		|(\gamma n+\delta)!|_p^{\gamma n+\delta -1} |x|_p^{\gamma n+\delta}=\left(p^{-\frac{(\gamma n+\delta)-s_p(\gamma n+\delta)}{(p-1)} \frac{(\gamma n+\delta-1)}{(\gamma n+\delta)}} \right)^{\gamma n+\delta} \to 0 \ \text{as} \ n \to \infty,
	\end{align}
	which is valid for any $p$ and all $x \in \mathbb{Q}_p$. Hence the power series converges for all $x \in \mathbb{Q}_p$ and for all $p$. From $(4)$ and $(8)$, we conclude that the given power series converges everywhere in $\mathbb{R}$ as well as in $\mathbb{Q}_p$.
\end{proof}
\begin{cor}
	The power series $\phi_{\gamma, \delta}^{\epsilon, q}=\sum_{n=0}^{\infty} \epsilon^n \frac{((\gamma n+\delta)!)^{\gamma n+\delta}}{q+((\gamma n+\delta)!)^{P_k(n)(\gamma n+\delta})} \frac{x^{\gamma n+\delta}}{(\gamma n+\delta)!}$ converges everywhere in $\mathbb{R}$ and in $\mathbb{Q}_p$, where $P_k(n)$ is a polynomial in $n$ of degree $k$. 
\end{cor}
Starting \ with \ above \ series  \ $(3)$, we derive a summation formula giving rational sum. 
\begin{thm}
	The summation formula 
	\begin{equation}
	\begin{aligned}
	& \sum_{n=0}^{\infty} ((\gamma n+\delta)!)^{\gamma n+\delta -1} x^{\gamma n} \left\{\frac{[(\gamma (n+1)+\delta)!]^{\gamma} (\gamma n+\delta +1)_{\gamma}^{\gamma n+\delta -1}}{q+[(\gamma (n+1)+\delta)!]^{N(\gamma(n+1)+\delta)}} x^{\gamma}- \frac{1}{q+((\gamma n+\delta)!)^{N(\gamma n+\delta)}} \right\} \\
	&=-\frac{(\delta !)^{\delta -1}}{q+(\delta !)^{N \delta}} ,
	\end{aligned}
	\end{equation}
	where $(\gamma n+\delta+1)_{\gamma}=(\gamma n+\delta+1)(\gamma n+\delta+2) \cdots (\gamma n+\delta+\gamma)$, is valid for all non-zero $x \in \mathbb{R}$ as well as for all non-zero $x \in \mathbb{Q}_p$ for every prime $p$.
\end{thm}
\begin{proof}
	From $(3)$, we have 
	\begin{align*}
		\phi_{\gamma, \delta}^{\epsilon, q}=\sum_{n=0}^{\infty} \epsilon^n \frac{((\gamma n+\delta)!)^{\gamma n+\delta}}{q+((\gamma n+\delta)!)^{N(\gamma n+\delta})} \frac{x^{\gamma n+\delta}}{(\gamma n+\delta)!}, \ \epsilon=\pm 1.
	\end{align*}
	For $\epsilon=-1$, 
	\begin{equation}
	\begin{aligned}
	\phi_{\gamma, \delta}^{-1, q} &=\sum_{n=0}^{\infty} (-1)^n \frac{((\gamma n+\delta)!)^{\gamma n+\delta}}{q+((\gamma n+\delta)!)^{N(\gamma n+\delta})} \frac{x^{\gamma n+\delta}}{(\gamma n+\delta)!} \\
	&= \frac{(\delta !)^{\delta}}{q+(\delta !)^{N \delta}} \frac{x^{\delta}}{\delta !}-\frac{((\gamma+\delta)!)^{\gamma+\delta}}{q+((\gamma+\delta)!)^{N(\gamma+\delta)}}\frac{x^{\gamma+\delta}}{(\gamma+\delta)!}+\frac{((2\gamma+\delta)!)^{2\gamma+\delta}}{q+((2\gamma+\delta)!)^{N(2\gamma+\delta)}}\frac{x^{2\gamma+\delta}}{(2\gamma+\delta)!}-\cdots 
	\end{aligned}
	\end{equation}
	Note that, 
	\begin{equation}
	\begin{aligned}
	& ((2 \gamma+\delta)!)^{\gamma}(\gamma+\delta+1)_{\gamma}^{\gamma+\delta-1}((\gamma +\delta)!)^{\gamma+\delta-1} \\ &=((\gamma+\delta)!)^{\gamma+\delta-1}((2 \gamma+\delta)!)^{\gamma}[(\gamma+\delta+1)(\gamma+\delta+2) \cdots (\gamma+\delta+\gamma)]^{\gamma+\delta-1} \\ &= ((2 \gamma+\delta)!)^{\gamma}[((\gamma+\delta)!)(\gamma+\delta+1)(\gamma+\delta+2) \cdots (\gamma+\delta+\gamma)]^{\gamma+\delta-1} \\ &=((2 \gamma+\delta)!)^{\gamma} [(\gamma+\delta+\gamma)!]^{\gamma+\delta-1} \\ &=((2 \gamma+\delta)!)^{\gamma} [(2 \gamma+\delta)!]^{\gamma+\delta-1} \\ &=((2 \gamma+\delta)!)^{2 \gamma+\delta-1}
	\end{aligned}
	\end{equation}
	Using equation $(11)$, the equation $(10)$ can be rewritten as 
	\begin{equation}
	\begin{aligned}
	& \phi_{\gamma, \delta}^{-1, q}=\frac{(\delta !)^{\delta-1}}{q+(\delta !)^{N \delta}} x^{\delta}+((\gamma+\delta)!)^{\gamma+\delta-1} x^{\gamma+\delta} \left[\frac{((2 \gamma+\delta)!)^{\gamma} (\gamma +\delta+1)_{\gamma}^{\gamma+\delta-1}}{q+((2 \gamma+\delta)!)^{N(2 \gamma +\delta)}}x^{\gamma}-\frac{1}{q+((\gamma+\delta)!)^{N(\gamma+\delta)}} \right]+\cdots 
	\end{aligned}
	\end{equation}
	\begin{equation}
	\begin{aligned}
	& -\phi_{\gamma, \delta}^{-1, q}=(\delta !)^{\delta-1} x^{\delta} \left[\frac{((\gamma +\delta)!)^{\gamma} (\delta+1)_{\gamma}^{\delta-1}}{q+((\gamma+\delta)!)^{N(\gamma+\delta)}}x^{\gamma}-\frac{1}{q+(\delta !)^{N \delta}}  \right] \\ & ((2 \gamma+\delta)!)^{2 \gamma+\delta-1} x^{2 \gamma+\delta} \left[\frac{((3 \gamma+\delta)!) (2 \gamma+\delta+1)_{\gamma}^{2 \gamma+\delta-1}}{q+((3 \gamma+\delta)!)^{N(3 \gamma+\delta)}}x^{\gamma}-\frac{1}{q+((2 \gamma+\delta)!)^{N(2 \gamma+\delta)}} \right]+\cdots
	\end{aligned}
	\end{equation}
	Adding $(12)$ and $(13)$ following a division by $x^{\delta}, \ (x \neq 0)$ results the formula
	\begin{equation}
	\begin{aligned}
	& \sum_{n=0}^{\infty} ((\gamma n+\delta)!)^{\gamma n+\delta-1} x^{\gamma n} \left\{\frac{[(\gamma (n+1)+\delta)!]^{\gamma} (\gamma n+\delta+1)_{\gamma}^{\gamma n+\delta-1}}{q+[(\gamma (n+1)+\delta)!]^{N(\gamma(n+1)+\delta)}}x^{\gamma} -\frac{1}{q+((\gamma n+\delta)!)^{N(\gamma n+\delta)}} \right\} \\ &=-\frac{(\delta !)^{\delta-1}}{q+(\delta !)^{N \delta}}
	\end{aligned}
	\end{equation}
	This proves the  summation formula, which gives rational sum and has a place for all $ x (\neq 0) \in \mathbb{R}$ as well as for all $x (\neq 0) \in \mathbb{Q}_p$.
	\end{proof}

\subsection{Adelic Aspect}
\begin{defn} \cite{schi-1} An adele is an infinite sequence 
	\begin{equation}
	a=(a_{\infty}, a_2, \cdots, a_p, \cdots ),
	\end{equation}
	in which $a_{\infty}$ is a real number, and $a_p, \ p \in \{2,3,5, \cdots\}$ a p-adic number and all but finite number of $a_p \in \mathbb{Z}_p$. The collection $\mathfrak{A}$ of all such sequences or adeles forms a ring under componentwise addition and multiplication. This is called \textit{ring of adeles} and its additive group is called the \textit{group of adeles.} \\ The elements of the ring of adeles $\mathfrak{A}$ that have an inverse are called \textit{ideles}. The set $\mathfrak{A}^{*}$ of all ideles forms a group under multiplication. It is called the group of \textit{ideles.}\\
	Thus the elements of the group of ideles are sequences 
	$$ \lambda=(\lambda_{\infty}, \lambda_2, \cdots, \lambda_p, \cdots),$$
	where $\lambda_p \neq 0$ and $|\lambda_p|_p=1$ for all $p$ with a finite number of exceptions. \\ Therefore an idele is an adele but the converse is not true i.e., 
	\begin{equation}
	\mathfrak{A} \subsetneq \mathfrak{A}^{*}.
	\end{equation}
	There is a topology in the group of adeles defined in the following way. Consider the subgroup $\mathfrak{A}^{\circ}$ of the adeles 
	\begin{equation}
	a=(a_{\infty}, a_2, \cdots, a_p, \cdots),
	\end{equation} where all $a_p \in \mathbb{Z}_p$. In $\mathfrak{A}^{\circ}$ we introduce the topology of the Tichonoff product of the topological spaces $\mathbb{R}, \ O_2, \cdots, O_p, \cdots,$ where $O_p$ is the subgroup of p-adic integers. This subgroup $\mathfrak{A}^{\circ}$ is declared to be an open set in $\mathfrak{A}$. \\
	Thus, a sequence of adeles $a^{(n)}=(a_{\infty}^{(n)}, a_{2}^{(n)}, \cdots, a_{p}^{(n)}, \cdots)$ is said to converge to the adele $a=(a_{\infty}, a_2, \cdots, a_p, \cdots)$ if it converges to $a$ componentwise and if there is an $m \in \mathbb{N}$ such that for $n \geq m$ the numbers $a_p-a_p^{(n)}$ are p-adic integers.
\end{defn}
We note down the following Lemma which will be used in Theorem (4.13).
\begin{lem}
	If for $a,b \in \mathbb{Q}_p$ we have $|a|_p \neq |b|_p$, then $|a+b|_p=\max \{|a|_p, |b|_p\}.$
\end{lem}
\begin{proof}
	Given $|a|_p \neq |b|_p$, i.e., either $|a|_p>|b|_p$ or $|a|_p<|b|_p$. \\ Now let us consider the case when $|a|_p>|b|_p$, then 
	\begin{align*}
		|a|_P=|(a+b)-b|_p & \leq \max \left(|a+b|_p, |b|_p\right) \\ &=|a+b|_P,
	\end{align*} for otherwise, $|a|_p \leq |b|_p$, a contradiction to our assumption. Therefore, $$ |a|_p \leq |a+b|_p \leq \max \left(|a|_p, |b|_p\right)=|a|_p,$$ and hence $|a+b|_p=\max \left(|a|_p,|b|_p\right).$ 
\end{proof}
Now recall the power series $F(x)=\sum_{n=1}^{\infty} \binom{a}{n}x^n$, which is convergent with respect to the p-adic norm and real norm which sums a rational number under certain conditions (see Theo. ($4.7$)). 
\begin{thm}
	Let us have the following sequence 
	\begin{equation}
	\mathcal{F}(x)=\left(F(x_{\infty}), F(x_{2}), \cdots, F(x_p), \cdots \right).
	\end{equation}
	If $x=(x_{\infty}, x_2, \cdots, x_p, \cdots)$ is an adele with $|x_{v}|_p<1, \ v \in \{\infty,2,3, \cdots, p, \cdots \}$ and $ a \in \mathbb{Z}_p$ i.e, $|a|_p \leq 1$, then  $\mathcal{F}(x)$ is an adele. 	
\end{thm} 
\begin{proof}
	We have 
	\begin{equation}
	F(x)=\sum_{n=1}^{\infty} \binom{a}{n}x^n.
	\end{equation}
	Let us define $F_n(x)=\sum_{i=1}^{n} \binom{a}{i}x^i$, then $F(x)=\lim_{n \to \infty} F_n(x)$. \\
	Since $x=(x_{\infty}, x_2, \cdots, x_p, \cdots)$ is an adele, $|x_v|_p<1, \ v \in \{\infty,2,3, \cdots, p, \cdots \}$ for all $v$ with a finite number of exceptions. However we assumed that $|x_v|_p<1$ for all $v \in \{\infty, 2,3, \cdots, p, \cdots \}$. So the power series $(19)$ converges. \\
	The Lemma (4.2) says if $a \in \mathbb{Z}_p$, then $ \binom{a}{n} \in \mathbb{Z}_p$ and hence $ \left| \binom{a}{n}\right|_p \leq 1$. \\  From $(19)$ and using Lemma (4.2), we have
	
	\begin{align*}
		|F_1(x_{\infty})|_p=|a x_{\infty}|_p=|a|_p|x_{\infty}|_p<1,
	\end{align*}, 
	\begin{align*}
		|F_2(x_{\infty})|_p=|ax_{\infty}+\binom{a}{2} x_{\infty}^2|_p=|x_{\infty}|_p \left|a+\binom{a}{2} x_{\infty}\right|_p=|x_{\infty}|_p|a|_p<1,
	\end{align*}
	\begin{eqnarray*}
		............& ............ 
	\end{eqnarray*}
	\begin{align*}
		|F_n(x_{\infty})|=& \left|ax_{\infty}+\binom{a}{2}x_{\infty}^2+\cdots+\binom{a}{n} x_{\infty}^n\right|_p \\ =&|x_{\infty}|_p \ \left|a+\binom{a}{2}x_{\infty}+\cdots+\binom{a}{n} x_{\infty}^{n-1} \right|_p \\ =&|x_{\infty}|_p|a|_p<1 , 
	\end{align*}
	Continuing in this one can show that 
	\begin{equation}
	|F(x_{\infty})|_p=\lim_{n \to \infty}|F_n(x_{\infty})|_p<1.
	\end{equation}
	Next for $p \in \{2,3, \cdots, p, \cdots \},$ \\
	\begin{align*}
		|F_1(x_{p})|_p=|a x_{p}|_p=|a|_p|x_{p}|_p<1,
	\end{align*}, 
	\begin{align*}
		|F_2(x_{p})|_p=|ax_{p}+\binom{a}{2} x_{p}^2|_p=|x_{p}|_p \left|\binom{a}{2} x_{p}\right|=|x_{p}|_p|a|_p<1,
	\end{align*}
	\begin{eqnarray*}
		............& ............ 
	\end{eqnarray*}
	\begin{align*}
		|F_n(x_{p})|=& \left|ax_{p}+\binom{a}{2}x_{p}^2+\cdots+\binom{a}{n} x_{p}^n\right|_p \\ =&|x_{p}|_p \ \left|a+\binom{a}{2}x_{p}+\cdots+\binom{a}{n} x_{p}^{n-1} \right|_p \\ =&|x_{p}|_p|a|_p<1 , 
	\end{align*}
	Continuing in this one can show that 
	\begin{equation}
	|F(x_{p})|_p=\lim_{n \to \infty}|F_n(x_{p})|_p<1.
	\end{equation}
	Thus $\mathcal{F}(x)$ is an adele whenever $x=(x_{\infty}, x_2, \cdots, x_p, \cdots)$ is an adele with all $|x|_p<1, \ p \in \{\infty,2,3, \cdots,p, \cdots \}$. 
	
\end{proof}
\begin{thm}
	The sequence $\mathcal{G}(x)=(F(a,b;c;x_{\infty}), F(a,b;c; x_2), \cdots, F(a,b;c;x_p), \cdots )$ is an idele and hence an adele if $x=(x_{\infty}, x_2, \cdots, x_p , \cdots)$ is an adele with $|x_v|_p<1$ for all $v \in \{\infty, 2,3, \cdots, p, \cdots \}$ and $a \in \mathbb{Z}_p^{\times}$, $b=c$, $p \in \{2,3,5, \cdots \}$, where 
	\begin{align}
		F(a,b;c;x)=\sum_{n=0}^{\infty} \frac{a(a+1) \cdots (a+n) b(b+1) \cdots (b+n)}{c(c+1) \cdots (c+n) n!} x^n=\sum_{n=0}^{\infty} \frac{(a)_n (b)_n}{(c)_n n!}x^n
	\end{align} is the hypergeometric series and $(a)_n=a(a+1) \cdots (a+n-1)$ is the Pochhammer symbol.
	
\end{thm}
\begin{proof} 	Since $|x_v|_p<1, \ v \in \{\infty, 2,3, \cdots, p, \cdots \}$, the power series
	$$F(a,b;c;x)=\sum_{n=0}^{\infty} \frac{(a)_n (b)_n}{(c)_n n!}x^n.$$ 
  converges.\\ 
	Putting $b=c$ in the above equation, we get
	$$F(a,b;b;x)=\sum_{n=0}^{\infty} \frac{(a)_n}{n!}x^n =(1+x)^a.$$
	Given that $x=(x_{\infty}, x_2, \cdots, x_p, \cdots)$ is an adele and $|x_v|_p<1$ for all $x_v$. Now using Lemma (4.10), $$ |F(a,b;b;x_{\infty})|_p=|(1+x_{\infty})^a|_p=|1+x_{\infty}|_p^a=\max \left(1, |x_{\infty}|_p^a\right)=1,$$
	$$|F(a,b;b;x_p)|_p=|(1+x_p)^a|_p=|1+x_p|_p^a=\max \left(1, |x_p|_p^a\right)=1, \ \forall p \in \{2,3,5, \cdots\}.$$
	Hence $(F(a,b;c;x_{\infty}), F(a,b;c; x_2), \cdots, F(a,b;c;x_p), \cdots )$ is an idele and by equation $(16)$, it is also adele.
\end{proof}
At the end we conclude the section by giving a general theorem relating p-adically convergent power series and adele. 
\begin{thm}
	If $f(x) \in \mathbb{Z}_p[[x]]$, then the sequence $(f(x_{\infty}), f(x_2), \cdots, f(x_p), \cdots)$ is an adele if $x=(x_{\infty}, x_2, \cdots, x_p, \cdots)$ is an adele.  
\end{thm}
\begin{proof}
Since $x=(x_{\infty}, x_2, \cdots, x_p, \cdots)$ is an adele, $|x_p|_p \leq 1$ for all $p$ with a finite number of exceptions. \\
If $f(x) \in \mathbb{Z}_p[[x]]$, then Lemma $((4.1))$ shows that $f(x)$ converges in $D(1^{-})=\{x \in \mathbb{Q}_p: |x|_p <1 \}$. Therefore,
$$ |f(x_p)|_p \leq 1, \ \forall p \in \{2,3,5, \cdots, p, \cdots\}.$$
Hence the sequence $(f(x_{\infty}), f(x_2), \cdots, f(x_p), \cdots)$ is an adele. 
\end{proof}

\section{Conclusion} We have investigated rationality of power series both in real norm as well in p-adic norm. Then we have studied the adelic aspect and some results with the help of rational summable power series.   \\ \\
\textbf{Acknowledgement:} The authors are grateful to Professor Neal Koblitz, University of Washington, United States. The authors are grateful to Professor Branko G. Dragovich, Institute of Physics, Belgrade, Yugoslavia, for his continuous help and support throughout the preparation of this paper.
The second author is grateful to The Council Of Scientific and Industrial Research (CSIR), Government of India, for the award of JRF (Junior Research Fellowship).

\end{document}